\documentclass[11pt,reqno]{amsart}
\usepackage{longtable}

\usepackage{tikz}
\usepackage{amsmath,amsthm,amsfonts,amssymb,latexsym,mathrsfs,color,extarrows}

\usepackage{subcaption}
\usetikzlibrary{arrows.meta}

\newtheorem{Theorem}{Theorem}[section]

\newtheorem{Proposition}[Theorem]{Proposition}

\newtheorem{Lemma}[Theorem]{Lemma}
\newtheorem{Definition}[Theorem]{Definition}
\newtheorem{Example}[Theorem]{Example}

\allowdisplaybreaks

\DeclareMathOperator{\Sym}{\mathfrak{S}}
 
\DeclareMathOperator{\des}{des}

\DeclareMathOperator{\inv}{inv}

\DeclareMathOperator{\maj}{maj}
\DeclareMathOperator{\Des}{Des}

\numberwithin{equation}{section}

\usepackage[backref=page]{hyperref}

\usepackage{hyperref}  
\hypersetup{
    colorlinks=true,
    urlcolor=black
}
\makeatletter
\def\thm@space@setup{%
  \thm@preskip=8pt \thm@postskip=8pt 
}
\makeatother

\def\RS{\mathop{\rm RS}\nolimits}
\def\AndI{\mathop{\rm And}\nolimits^{I}}  
\def\AndII{\mathop{\rm And}\nolimits^{I\!I}} 

\newcommand{\qbinom}[2]{\left[ \genfrac{}{}{0pt}{}{#1}{#2} \right]_q}

\title[On the distributions of the statistics (des, maj, inv)]{On the distributions of the statistics (des, maj, inv) over several classes of permutations}

\author{Guo-Niu Han}
\address{I.R.M.A., UMR 7501, Universit\'e de Strasbourg et CNRS, 7 rue
Ren\'e Descartes, F-67084 Strasbourg, France}
\email{guoniu.han@unistra.fr}

\author{Kathy Q. Ji}
\address{ Center for Applied Mathematics and KL-AAGDM,
Tianjin University,
Tianjin 300072, P.R. China
}
\email{kathyji@tju.edu.cn}

\author{Huan Xiong}
\address{
 Institute for Advanced Study in Mathematics, 
   Harbin Institute of Technology,
   Heilongjiang 150001, P.R. China
}
\email{huan.xiong.math@gmail.com}

\makeatletter
\@namedef{subjclassname@2020}{%
	\textup{2020} Mathematics Subject Classification}
\makeatother

\date{2026/06/03}
    
\begin{document}

	\begin{abstract}
    We investigate the joint distribution of the trivariate statistics $(\des, \maj, \inv)$ on classical permutations, André permutations of the first and second kinds, and Simsun permutations. By decomposing permutations according to the position of the smallest element, we obtain explicit recurrence relations for the generating functions of these statistics.  In the classical permutation setting, our recurrence relation yields the generating function for the trivariate statistics $(\des, \maj, \inv)$ due to Gessel, which is typically proved using MacMahon’s technique.

	\end{abstract}

\maketitle

\section{Introduction}

The study of permutation statistics is a popular topic in enumerative combinatorics. Since the pioneering work of MacMahon \cite{macmahon1913indices}, the distribution of statistics such as the descent number ($\des$), major index ($\maj$), and inversion number ($\inv$) has been extensively investigated \cite{Foata-Sch-1970, Garsia1979, gessel1977generating}.

Let $\Sym_n$ denote the symmetric group on $\{1,2,\dots,n\}$. For a permutation $\sigma = \sigma_1\sigma_2\cdots\sigma_n\in\Sym_n$, the \textit{inversion number} is defined by
\begin{align*}
\inv(\sigma) &= \#\bigl\{(i,j)\mid 1 \le i < j \le n,\ \sigma_i > \sigma_j\bigr\},
\end{align*}
where the symbol $\#$ stands for the cardinality of a set.

The descent set of $\sigma$, denoted by $\Des(\sigma)$, is defined by
\begin{equation*}
\Des(\sigma) = \bigl\{i \mid 1 \le i \le n-1,\ \sigma_i > \sigma_{i+1}\bigr\},
\end{equation*}
and the \textit{descent number} $\des(\sigma)$ satisfies
\begin{equation*}
\des(\sigma) = \#\Des(\sigma).
\end{equation*}
The \textit{major index} $\maj(\sigma)$ is defined as the sum of its descent positions, namely
\begin{equation*}
\maj(\sigma)=\sum_{i \in \Des(\sigma)} i.
\end{equation*}

The bivariate generating function 
\begin{equation*}
A_n^{\des,\inv}(t,q)=\sum_{\sigma\in\Sym_n}t^{\des(\sigma)+1}q^{\inv(\sigma)},
\end{equation*} 
which records descents and inversions, has a well-known exponential generating function due to Stanley \cite{Stanley-1976}: 
\[
\sum_{n\geq0}A_n^{\des,\inv}(t,q)\frac{x^n}{[n]_q!}=\frac{1-t}{1-te(x(1-t);q)},
\]
where
\begin{align}
 e(z;q)
 &= \sum_{n\ge 0} \frac{z^n}{[n]_q!}.  
\end{align}
Here and in the sequel, for a positive integer $n$, we define
\[
[n]_q := \frac{1-q^n}{1-q} = 1+q+\cdots+q^{n-1}
\]
and for $n \ge 1$,
\[
[n]_q! := [1]_q[2]_q\cdots[n]_q.
\]
Assume that $[0]_q! = 1$.

Note that in the definition of $A_n^{\des,\inv}(t,q)$, 
the power of $t$ is 
$\des(\sigma)+1$, while some articles may adopt $\des(\sigma)$. 
This elegant formula unified the Eulerian and Mahonian statistics in a single generating series. {Stanley's proof} used binomial posets and Möbius inversion, establishing a foundational link between poset theory and permutation enumeration. {Garsia \cite{Garsia1979} and {Gessel \cite{gessel1977generating}} later provided alternative proofs. Subsequently, {D\'esarm\'enien and Foata \cite{desarmenien1992signed} observed that the right-hand side of Stanley's identity can be written as 
\[
\biggl(1-t\sum_{n\geq1}(1-t)^{n-1}\frac{x^n}{[n]_q!}\biggr)^{-1},
\]
from which they derived a ``semi'' \(q\)-recurrence: For $n\geq 1$
\begin{equation}\label{desarfoarec}
A_n^{\des,\inv}(t,q)=t(1-t)^{n-1}+\sum_{i=1}^{n-1}
\qbinom{n}{i} A_i^{\des,\inv}(t,q)t(1-t)^{n-1-i},
\end{equation}
where  the $q$-binomial coefficient (also called the Gaussian polynomial) is defined by
\[
 {n  \brack i}_q =
\begin{cases}
\dfrac{(1 - q^{n})(1 - q^{n-1}) \cdots (1 - q^{n-i+1})}{(1 - q^i)(1 - q^{i-1}) \cdots (1 - q)}, & \text{for } 0\leq i\leq n , \\[1em]
0. & \text{otherwise},
\end{cases}
\]
see Andrews \cite[Chapter~1]{Andrews-1976}. 

Note that the summand in \eqref{desarfoarec} mixes \(q\)-dependent and \(q\)-independent factors. A ``fully'' \(q\)-recurrence---in which every term depends explicitly on \(q\)---was later given by Chow \cite{chow2010recurrence}:
\begin{align}\label{eq: chow}
A_{n+1}^{\des,\inv}(t,q)=(1+tq^n)A_n^{\des,\inv}(t,q)+\sum_{k=1}^{n-1}\qbinom{n}{k} q^k A_{n-k}^{\des,\inv}(t,q)A_k^{\des,\inv}(t,q), 
\end{align}
which reduces to the classical convolution recurrence for Eulerian polynomials when \(q=1\).

Among various permutation classes, Andr\'e permutations hold particular significance due to their intimate connection with Euler numbers. The Euler numbers $E_n$, which appear in the Taylor expansions of $\sec(x) + \tan(x)$, were shown by Andr\'e in the late 19th century to enumerate alternating permutations \cite{andre1881permutations}. Andr\'e permutations, introduced by Foata and 
Schützenberger~\cite{Foata-Sch-1970} and further studied by Strehl \cite{Str74} and Foata and Strehl \cite{FSt74, FSt76}, provide additional combinatorial interpretations for these numbers.

For a permutation $\sigma = \sigma_1\sigma_2\cdots\sigma_n$ with $n \geq 2$, consider its decomposition $\sigma = \tau\min(\sigma)\tau'$, where $\min(\sigma)$ is the minimum element. Then $\sigma$ is called an \textit{Andr\'e permutation of the first kind} if both $\tau$ and $\tau'$ are Andr\'e permutations of the first kind, and the maximum element of the concatenated subword $\tau\tau'$ lies in $\tau'$. Similarly, $\sigma$ is an \textit{Andr\'e permutation of the second kind} if both $\tau$ and $\tau'$ are Andr\'e permutations of the second kind, and the minimum element of $\tau\tau'$ lies in $\tau'$. The empty permutation and single-element permutations are defined to be both types of Andr\'e permutations. Denoting by $\AndI_n$ and $\AndII_n$ the sets of Andr\'e permutations of the first and second kind on $[n] = \{1,2,\ldots,n\}$ respectively, it is well known that $E_n = |\AndI_n| = |\AndII_n|$.

Another important class is that of Simsun permutations, introduced by Simion and Sundaram in their study of homology representations of the symmetric group \cite{sundaram1995homology,sundaram1996plethysm}. A permutation $\sigma = \sigma_1\sigma_2\ldots\sigma_n$ on $[n]$ is called a \textit{Simsun permutation} if it contains no double descents, and this property persists after successively removing the largest remaining elements $n, n-1, \ldots, 2, 1$ in order. For example, it is easy to see that $\sigma= 2147365$ is a Simsun permutation since $2147365$, $214365$, $21435$, $2143$, $213$, $21$, $1$ have no double descents. Remarkably, the number of Simsun permutations on $[n]$, denoted $\RS_n$, also equals the Euler number $E_{n+1}$. The notation $\RS_n$ was first adopted by Chow and Shiu \cite{chow2011counting}, who established a bijection between Andr\'e permutations of the first kind and Simsun permutations, showing that the descent statistic is equidistributed over these two classes:

\[
\sum_{\sigma \in \AndI_n} t^{\mathrm{des}(\sigma)} = \sum_{\sigma \in \RS_{n-1}} t^{\mathrm{des}(\sigma)}.
\]

Beyond their enumerative significance, Andr\'e permutations and Simsun permutations play crucial roles in understanding the $cd$-indices of simplicial Eulerian posets \cite{bayer2019cd, He96,  
purtill1993andre,stanley1994flag}. Their properties have been further explored by Barnabei et al. \cite{barnabei2020permutations}, Chow and Shiu \cite{chow2011counting}, Deutsch and  Elizalde~\cite{deutsch2012restricted}, Disanto \cite{Di14}, and Foata and Han \cite{FH01}, among others.

\medskip

In this paper, we extend the study of permutation statistics to the trivariate setting for these important permutation classes. 
We investigate the joint distribution of the statistics $(\mathrm{des}, \mathrm{maj}, \mathrm{inv})$ for classical permutations, Andr\'e permutations of the first and second kind, and Simsun permutations. More precisely, 
\begin{align}
A_n(t,q,p)&= \sum_{\sigma\in \Sym_n} t^{\mathrm{des}(\sigma)} q^{\mathrm{inv}(\sigma)} p^{\mathrm{maj}(\sigma)},\\[5pt]
A_n^I(t,q,p)&= \sum_{\sigma\in \AndI_n} t^{\mathrm{des}(\sigma)} q^{\mathrm{inv}(\sigma)} p^{\mathrm{maj}(\sigma)},\\[5pt]
A_n^{I\!I}(t,q,p)&= \sum_{\sigma\in \AndII_n} t^{\mathrm{des}(\sigma)} q^{\mathrm{inv}(\sigma)} p^{\mathrm{maj}(\sigma)},\\[5pt]
A_n^{RS}(t,q,p)&= \sum_{\sigma\in \RS_n} t^{\mathrm{des}(\sigma)} q^{\mathrm{inv}(\sigma)} p^{\mathrm{maj}(\sigma)}. 
\end{align}
By utilizing decomposition techniques based on the position of the smallest element, we derive exact recurrence relations for the generating functions of these statistics.

\begin{Theorem} \label{thm:classical_rec} We have 
\begin{itemize}
\item[ (a).]  $A_0(t,q,p) = A_1(t,q,p) = 1$.  For $n \ge 2$,
\begin{equation} \label{eq:classical_rec}
A_n(t,q,p) = A_{n-1}(tp, q, p) + t \sum_{j=1}^{n-1} (pq)^j \qbinom{n-1}{j} A_j(t,q,p) A_{n-1-j}(tp^{j+1}, q, p).
\end{equation}

\item[ (b).] $A^I_0(t,q,p) = A^I_1(t,q,p) = 1$.  For $n\geq 2$, 
\begin{equation}  \label{eq:AndreI}
A_n^I(t,q,p) = A_{n-1}^I(tp, q, p) + t \sum_{j=1}^{n-2} (pq)^j \qbinom{n-2}{j} A_j^I(t,q,p) A_{n-1-j}^I(tp^{j+1}, q, p).
\end{equation}

\item[ (c).] $A^{I\!I}_0(t,q,p) = A^{I\!I}_1(t,q,p) = 1$.  For $n\geq 2$, 
\begin{equation} \label{eq:AndreII}
A_n^{II}(t,q,p) = A_{n-1}^{II}(tp, q, p) + t \sum_{j=1}^{n-2} (pq^2)^j \qbinom{n-2}{j} A_j^{II}(t,q,p) A_{n-1-j}^{II}(tp^{j+1}, q, p).
\end{equation}

\item[ (d).] $A^{RS}_0(t,q,p) = A^{RS}_1(t,q,p) = 1$.  For $n\geq 2$, 
\begin{equation} \label{defi:simsun}
A_n^{RS}(t,q,p) = A_{n-1}^{RS}(tp, q, p) + t \sum_{j=1}^{n-1} (pq)^j \qbinom{n-1}{j} A_j^{II}(t,q,p) A_{n-1-j}^{RS}(tp^{j+1}, q, p).
\end{equation}  
\end{itemize}
\end{Theorem}

Note that the $p=1$ case in \eqref{eq:classical_rec} is equivalent to \eqref{eq: chow}. Thus, our result is a generalization of Chow's result from the bivariate generating function to the trivariate case. The summation in \eqref{defi:simsun} involves $A_j^{II}$ (Andr\'e II) and $A_{n-1-j}^{RS}$ (Simsun).

 \medskip 

 In the classical permutation setting, our recurrence relation yields the generating function for the trivariate statistics $(\des, \maj, \inv)$ due to Gessel, which is typically proved using MacMahon’s technique \cite{gessel1977generating}.

\begin{Theorem}[Gessel \cite{gessel1977generating}] \label{thm:qpPoly:Sn2}
The following $q$-series identity for $A_n(t,q,p)$ holds:
\begin{equation}\label{eqn:qpPoly:Sn2}
    \sum_{n\geq 0} A_n(t,q,p) \frac{u^n}{(q;q)_n (t;p)_{n+1}} 
    = \sum_{m\geq 0} \frac{t^m}{ (u;q)_{\infty}(up;q)_{\infty}\cdots (up^m;q)_{\infty}}.      
\end{equation}
\end{Theorem}

Here and throughout, the $q$-shifted factorial is defined by
\[
\begin{aligned}
(u;q)_n &:=
  \begin{cases}
    1, & \text{if } n = 0, \\
    (1-u)(1-uq) \cdots (1-uq^{n-1}), & \text{if } n \ge 1,
  \end{cases} \\
(u;q)_\infty &:= \lim_{n \to \infty} (u;q)_n
  = \prod_{n \ge 0} (1 - uq^n).
\end{aligned}
\]

\medskip

In contrast, our approach does not currently yield generating functions for the corresponding trivariate statistics on Andr\'e permutations by means of the recurrences in \eqref{eq:AndreI} and \eqref{eq:AndreII}. 
An open problem is to derive such generating functions through alternative methodologies that have not yet been fully developed.

\section{Recurrence formulas}\label{sec:Sym}

In this section, we investigate trivariate generating functions for classical permutations, André I permutations, André II permutations, and Simsun permutations, and complete the proof of Theorem \ref{thm:classical_rec}.

All four recurrence relations derived in this paper originate from a unified combinatorial decomposition strategy: factor a permutation at the location of its minimal entry $1$, which yields the canonical decomposition map $\Psi$ defined below.

\begin{Definition}[Decomposition map $\Psi$]
Let $\sigma=\sigma_1\cdots \sigma_n \in \Sym_n$, and suppose   $\sigma_{j+1}=1$. Define the left segment $\sigma^L=\sigma_1\cdots \sigma_j$ and right segment $\sigma^R=\sigma_{j+2}\cdots \sigma_n$. The map $\Psi$ sends $\sigma$ to the triple $\Psi(\sigma)=(L,R,w)$, where:
\begin{itemize}
\item $L\in\Sym_j$ is the order-preserving reduction of the letters of $\sigma^L$ onto the alphabet $\{1,2,\dots,j\}$;
\item $R\in\Sym_{n-j-1}$ is the order-preserving reduction of the letters of $\sigma^R$ onto the alphabet $\{1,2,\dots,n-j-1\}$;
\item $w=w_1w_2\cdots w_n$ is a binary $(0,1)$-word of length $n$ specified by
\[
w_i=
\begin{cases}
0, & \text{if  } i \in \sigma^L,\\
1, & \text{otherwise}.
\end{cases}
\]
\end{itemize}
\end{Definition}

\begin{Example}
For $\sigma=6\,9\,2\,5\,8\,1\,4\,7\,3\in\Sym_9$, we have $\Psi(\sigma)=(L,R,w)$ with $L=3\,5\,1\,2\,4$, $R=2\,3\,1$, and binary string $w=1\,0\,1\,1\,0\,0\,1\,0\,0$.
\end{Example}

Given $\sigma\in\Sym_n$ with $\sigma_{j+1}=1$ and associated decomposition triple $\Psi(\sigma)=(L,R,w)$, it is not difficult to see that $\Psi$ is reversible and the following relations are satisfied: 
\begin{align}
\des(\sigma) &= \des(L) + \des(R) + 1-\delta_{j,0}, \label{eq:dec-des}\\
\maj(\sigma) &= \maj(L) + \maj(R) + (j+1)\des(R) + j, \label{eq:dec-maj}\\
\inv(\sigma) &= \inv(L) + \inv(R) + \inv(w), \label{eq:dec-inv}
\end{align}
where $\delta_{j,0}$ denotes the Kronecker delta ($\delta_{j,0}=1$ if $j=0$, and $\delta_{j,0}=0$ otherwise).

 We first recall a result due to MacMahon \cite{macmahon1960}. For $n \ge j \ge 0$, let $\mathcal{M}(j,n-j)$ be the set of $(0,1)$-sequences of length $n$ consisting of $j$ copies of $0$'s and $n-j$ copies of $1$'s. The following well-known result is due to MacMahon (see \cite[Chapter 3.4]{Andrews-1976}). 
For $n \ge j \ge 0$,
\begin{align}\label{MacMahona}
{n\brack j}_q &= \sum_{w\in \mathcal{M}(j,n-j)} q^{\operatorname{inv}(w)}.
\end{align}
In particular, let $\mathcal{M}_{1}(j,n-j)$ be the set of $(0,1)$-sequences $w=w_1w_2\cdots w_{n} \in \mathcal{M}(j,n-j) $ such that $w_{1}=1$. Then 
\begin{align}\label{MacMahonb}
q^{j}{n-1\brack j}_q &= \sum_{w\in \mathcal{M}_{1}(j,n-j)} q^{\operatorname{inv}(w)}.
\end{align}
Let $\mathcal{M}_{12}(j,n-j)$ be the set of $(0,1)$-sequences $w=w_1w_2\cdots w_{n} \in \mathcal{M}(j,n-j) $ such that $w_{1}=1$ and $w_{2}=1$. Then 
\begin{align}\label{MacMahonc}
q^{2j}{n-2\brack j}_q &= \sum_{w\in \mathcal{M}_{12}(j,n-j)} q^{\operatorname{inv}(w)}.
\end{align}
Let $\mathcal{M}_{1n}(j,n-j)$ be the set of $(0,1)$-sequences $w=w_1w_2\cdots w_{n} \in \mathcal{M}(j,n-j) $ such that $w_{1}=1$ and $w_{n}=1$. Then 
\begin{align}\label{MacMahond}
q^j{n-2\brack j}_q &= \sum_{w\in \mathcal{M}_{1n}(j,n-j)} q^{\operatorname{inv}(w)}.
\end{align}

\begin{Proposition}\label{mainprop}  For $n> j\geq 0$, 
let  $\Sym_{n,j}$ (resp. $\AndI_{n,j}$, $\AndII_{n,j}$, $\RS_{n,j}$)  denote the set of permutations $\sigma=\sigma_1\cdots \sigma_n $ in $\Sym_n$ (resp. $\AndI_{n}$, $\AndII_{n}$, $\RS_{n}$) such that $\sigma_{j+1}=1$. We have 
\begin{itemize}
\item[(a).] The decomposition map $\Psi$ is a bijection between the set $\Sym_{n,j}$ and the set $\Sym_j \times \Sym_{n-j-1}\times \mathcal{M}_1(j, n-j)$.

\item[(b).] 
The decomposition map $\Psi$ is a bijection between the set $\AndI_{n,j}$ and the set $\AndI_j \times \AndI_{n-j-1}\times \mathcal{M}_{1n}(j, n-j)$. 

\item[(c).] The decomposition map $\Psi$ is a bijection between the set $\AndII_{n,j}$ and the set $\AndII_j \times \AndII_{n-j-1}\times \mathcal{M}_{12}(j, n-j)$. 

\item[(d).] The decomposition map $\Psi$ is a bijection between the set $\RS_{n,j}$ and the set $\AndII_j \times \RS_{n-j-1}\times \mathcal{M}_{1}(j, n-j)$. 

\end{itemize}
\end{Proposition}

\begin{proof}
Assertions (a)--(c) follow directly from the definition of $\Psi$ and the respective combinatorial characterizations of classical, Andr\'e I, and Andr\'e II permutations. We focus the remainder of the proof on part (d).

Recall the standard characterization of Simsun permutations: $\sigma\in\Sym_n$ is Simsun ($\sigma\in\RS_n$) if and only if, for every integer $1\le k\le n$, the subsequence $\sigma^{(k)}$ formed by all entries of $\sigma$ bounded above by $k$ (preserving original positional order) contains no double descent. This successive-truncation characterization is equivalent to the defining property of $\RS_n$.

Existing literature establishes recursive constructions for Simsun permutations by inserting the \textit{maximal} element; see Sundaram \cite{sundaram1996plethysm} and Chow \& Shiu \cite{chow2011counting}. Our alternative recursion via the minimal element $1$ does not appear in prior references, so we supply a self-contained proof below.

Since $1$ is globally minimal, $1$ belongs to every truncated subsequence $\sigma^{(k)}$ for all $k\ge1$ and partitions each such subsequence as $\sigma^{(k)}=(L^{(k)},1,R^{(k)})$, where $L^{(k)}$ (resp. $R^{(k)}$) collects entries $\le k$ from the reduced left permutation $L$ (resp. reduced right permutation $R$). The subsequence $\sigma^{(k)}$ avoids double descents if and only if:
\begin{enumerate}
\item both $L^{(k)}$ and $R^{(k)}$ avoid double descents, and
\item $L^{(k)}$ does not terminate with a descent.
\end{enumerate}
Varying $k$ from $1$ to $n$, the family of constraints imposed on all $R^{(k)}$ is equivalent to $R\in\RS_{n-j-1}$. Meanwhile the full set of constraints on all $L^{(k)}$ condenses to the following property for $L$:
\begin{itemize}
\item[(P)] For every integer $k$, the subsequence of $L$ composed of entries $\le k$ has no double descent and does not end in a descent.
\end{itemize}
We now verify that a permutation satisfies property (P) if and only if it belongs to $\AndII$, proceeding by induction on permutation length $|w|$. Base cases $|w|\le1$ hold trivially.

Let $w=\alpha\mu\beta$ with $\mu=\min(w)$. Splitting each truncated subsequence $w^{(k)}=(\alpha^{(k)},\mu,\beta^{(k)})$, property (P) for $w$ decomposes into three separate requirements:
\begin{itemize}
\item[(i)] $\alpha$ satisfies property (P);
\item[(ii)] $\beta$ satisfies property (P);
\item[(iii)] for every $k$, nonemptiness of $\alpha^{(k)}$ forces nonemptiness of $\beta^{(k)}$.
\end{itemize}
Condition (iii) is equivalent to requiring the smallest element of $\alpha\cup\beta$ lies entirely inside $\beta$. By induction, (i) and (ii) imply $\alpha,\beta\in\AndII$. The combined conditions $\alpha,\beta\in\AndII$ plus $\min(\alpha\beta)\in\beta$ exactly match the standard recursive definition of second-kind André permutations, so $L\in\AndII_j$. This completes the proof of part (d).
\end{proof}

 With Proposition \ref{mainprop} in hand, we proceed to establish Theorem \ref{thm:classical_rec}.   

\medskip 
 \begin{proof}[Proof of Theorem \ref{thm:classical_rec}]
Parts (a)--(d) follow via uniform reasoning using Proposition \ref{mainprop} alongside identities \eqref{MacMahonb}--\eqref{MacMahond}; we present full details only for case~(a).

We classify $\sigma \in \Sym_n$ by the position $j+1$ of the entry $1$,  so $\sigma\in\Sym_{n,j}$ maps to $\Psi(\sigma)=(L,R,w)$.

\noindent\textbf{Case $j=0$:} Here $L=\emptyset$, the binary word $w$ is all-ones, and statistic simplifications from \eqref{eq:dec-des}--\eqref{eq:dec-inv} give
\begin{equation}
\sum_{\sigma\in \Sym_{n,0}} t^{\des(\sigma)} q^{\inv(\sigma)} p^{\maj(\sigma)}
=\sum_{R\in \Sym_{n-1}} t^{\des(R)} q^{\inv(R)} p^{\maj(R)+\des(R)}
=A_{n-1}(tp,q,p).
\end{equation}

\noindent\textbf{Case $j\ge1$:} Applying Proposition \ref{mainprop}(a) together with the three statistic decomposition formulas yields
\begin{align}
\sum_{\sigma\in \Sym_{n,j}} t^{\des(\sigma)} q^{\inv(\sigma)} p^{\maj(\sigma)}
&=tp^j\Bigl(\sum_{L \in \Sym_j}t^{\des(L)} q^{\inv(L)} p^{\maj(L)}\Bigr)\nonumber \\[4pt]
&\quad \times \Bigl(\sum_{R \in \Sym_{n-j-1}} t^{\des(R)} q^{\inv(R)} p^{\maj(R)+(j+1)\des(R)}\Bigr)\nonumber \\[4pt]
&\quad \times \Bigl(\sum_{ w \in \mathcal{M}_1(j, n-j)}q^{\inv(w)}\Bigr). \label{pfaaa}
\end{align}
Substitute  \eqref{MacMahonb}  into  \eqref{pfaaa},   we arrive at
\[
\sum_{\sigma\in \Sym_{n,j}} t^{\des(\sigma)} q^{\inv(\sigma)} p^{\maj(\sigma)}
=t(pq)^j{n-1 \brack j}_q\, A_j(t,q,p)\,A_{n-j-1}(tp^{j+1},q,p).
\] 
Adding the $j=0$ term and summing over $1\le j\le n-1$ gives \eqref{eq:classical_rec} in Theorem \ref{thm:classical_rec}(a).
\end{proof}

\section{Gessel's formula}

We now present the proof of the equivalence between the recurrence relation \eqref{eq:classical_rec} in Theorem \ref{thm:classical_rec} and the identity \eqref{eqn:qpPoly:Sn2}, which implies 
Theorem~\ref{thm:qpPoly:Sn2}. 

Let us define the normalized polynomial $B_n(t,q,p)$ as:
\begin{equation*}
    B_n(t,q,p)=\frac{A_n(t,q,p)}{(t;p)_{n+1}}.
\end{equation*}
Observe that the initial conditions are given by:
\begin{equation}\label{eqn:qpPoly:Snrea}
    B_0(t,q,p)=\frac{1}{1-t}, \quad B_1(t,q,p)=\frac{1}{(1-t)(1-tp)}.
\end{equation}
Substituting this transformation into the recurrence relation \eqref{eq:classical_rec} in Theorem~\ref{thm:classical_rec}, we obtain the following recurrence for $n\geq 2$:
\begin{align}\label{eqn:qpPoly:Snre}
    B_n(t,q,p) = & \frac{1}{1-t} B_{n-1}(tp,q,p) \nonumber \\
    & + t \sum_{j=1}^{n-1} (pq)^j \qbinom{n-1}{j} B_j(t,q,p) B_{n-1-j} (tp^{j+1},q,p).
\end{align}
Next, we expand $B_n(t,q,p)$ with respect to $t$:
\begin{equation*}
    B_n(t,q,p)=\sum_{m\geq 0} t^m B_{n,m}(q,p).
\end{equation*}
By equating coefficients of $t^m$ in \eqref{eqn:qpPoly:Snrea} and \eqref{eqn:qpPoly:Snre}, we obtain the initial values for the coefficients:
\begin{align}
    B_{0,m}(q,p) &= 1 \quad \text{for} \quad m\geq 0, \label{initval_B0}\\
    B_{1,m}(q,p) &= \frac{1-p^{m+1}}{1-p} \quad \text{for} \quad m\geq 0. \label{initval_B1}
\end{align}
For $n\geq 2$ and $m\geq 0$, the recurrence relation becomes:
\begin{align}\label{eqn:qpPoly:Snre3}
    B_{n,m}(q,p) &= \sum_{b=0}^m p^b B_{n-1,b}(q,p) \nonumber \\
    & \quad + \sum_{j=1}^{n-1} (pq)^j \qbinom{n-1}{j} \sum_{b=0}^{m-1} p^{(j+1)b} B_{j,m-b-1}(q,p) B_{n-1-j,b} (q,p).
\end{align}
According to Theorem \ref{thm:qpPoly:Sn2}, proving its main result is equivalent to proving:
\begin{equation}\label{eqn:qpPoly:Snre4}
    \sum_{n\geq 0}B_{n,m}(q,p)  \frac{u^n}{(q;q)_n} = \frac{1}{(u;q)_\infty (up;q)_\infty \cdots (up^m;q)_\infty}.
\end{equation}
Therefore, establishing the equivalence between Theorem \ref{thm:classical_rec} (a) and Theorem \ref{thm:qpPoly:Sn2} amounts to verifying the equivalence of \eqref{eqn:qpPoly:Snre3} and \eqref{eqn:qpPoly:Snre4}. To this end, we define the generating function $f_m$ and the product form $g_m$:
\begin{equation*}
    f_{m}(u;p,q) = \sum_{n\geq 0} B_{n,m}(q,p)\frac{u^n}{(q;q)_n},
\end{equation*}
\begin{equation*}
    g_{m}(u;p,q) = \frac{1}{(u;q)_\infty (up;q)_\infty \cdots (up^m;q)_\infty}.
\end{equation*}
We proceed to demonstrate that $f_{m}(u;p,q)$ and $g_{m}(u;p,q)$ satisfy the same recurrence relation and initial conditions.

\begin{Lemma} \label{lem:fm}
For $m\geq 1$, the function $f_m(u;p,q)$ satisfies the recurrence:
\begin{align} \label{recurra}
    f_{m}(u;p,q) &= f_{m}(uq;p,q) + up^m f_m(u;p,q) \nonumber \\
    & \quad +\sum_{b=0}^{m-1} up^b f_{b}(u;p,q) f_{m-1-b}(up^{b+1}q;p,q),
\end{align}
with the initial condition
\begin{equation} \label{recurrb}
    f_{0}(u;p,q)=\frac{1}{(u;q)_\infty}.
\end{equation}
\end{Lemma}

\begin{proof} 
By definition, $B_{n,0}(q,p)=1$ for all $n\geq 0$. The initial condition \eqref{recurrb} is immediate from the following  Euler's identity 
\[\sum_{n\geq 0} \frac{z^n}{(q;q)_n} = \frac{1}{(z;q)_\infty}.\] 

We next show the recurrence relation \eqref{recurra}. Observe that
\begin{align*}
    f_m(u;p,q) - f_m(uq;p,q)
    &= \sum_{n\geq 0} B_{n,m}(q,p) \frac{u^n}{(q;q)_n} - \sum_{n\geq 0} B_{n,m}(q,p) \frac{u^n q^n}{(q;q)_n} \\
    &= \sum_{n\geq 1} B_{n,m}(q,p) \frac{u^n }{(q;q)_{n-1}}.
\end{align*}
Using \eqref{initval_B1}, \eqref{eqn:qpPoly:Snre3}
and separating the $n=1$ term, we have:
\begin{align*}
    & u\frac{1-p^{m+1}}{1-p} + \sum_{n\geq 2} B_{n,m}(q,p) \frac{ u^n}{(q;q)_{n-1}} \\
    & = u\frac{1-p^{m+1}}{1-p} + \sum_{n\geq 2} \left(\sum_{b=0}^m p^b B_{n-1,b}(q,p)\right) \frac{u^n}{(q;q)_{n-1}} \\
    & \quad + \sum_{n\geq 2} \left(\sum_{j=1}^{n-1} (pq)^j \qbinom{n-1}{j} \sum_{b=0}^{m-1} p^{(j+1)b} B_{j,m-b-1}(q,p) B_{n-1-j,b} (q,p)\right)\frac{u^n}{(q;q)_{n-1}}.
\end{align*}
We split the summation into two parts corresponding to the two terms in the recurrence. The first part becomes:
\begin{equation*}
    \sum_{b=0}^m p^b \left(\sum_{n\geq 2} B_{n-1,b}(q,p)\frac{u^n}{(q;q)_{n-1}}\right) = \sum_{b=0}^m u p^b \left(f_{b}(u;p,q)-1\right).
\end{equation*}
The second part, involving the convolution, becomes:
\begin{align*}
    & \sum_{b=0}^{m-1} u p^b \sum_{j\geq 1} B_{j,m-b-1}(q,p) \frac{(u p^{b+1}q)^{j}}{(q;q)_j} \left(\sum_{n\geq j+1} B_{n-j-1,b} (q,p)\frac{u^{n-j-1}}{(q;q)_{n-j-1}}\right) \\
    &= \sum_{b=0}^{m-1} u p^b (f_{m-b-1}(up^{b+1}q;p,q)-1) f_{b}(u;p,q).
\end{align*}
Combining these, the terms involving $-1$ cancel out with the initial term $u(1-p^{m+1})/(1-p)$, yielding:
\begin{equation*}
    u p^m f_m(u;p,q) + \sum_{b=0}^{m-1} u p^b f_{b}(u;p,q) f_{m-1-b}(uqp^{b+1};p,q),
\end{equation*}
as desired. This completes the proof.
\end{proof}

By definition, it is easy to see that $g_{0}(u;p,q)=1/(u;q)_\infty$. It remains to show that $g_{m}(u;p,q)$ satisfies the same recurrence relation as $f_{m}(u;p,q)$.

\begin{Lemma} \label{lem:gm}
For $m\geq 1$,
\begin{align}\label{recurlab}
    g_{m}(u;p,q) &= g_{m}(uq;p,q) + u p^m g_m(u;p,q) \nonumber \\
    & \quad + \sum_{b=0}^{m-1} u p^b g_{b}(u;p,q) g_{m-1-b}(up^{b+1}q;p,q).
\end{align}
\end{Lemma}

\begin{proof} 
We observe the following relations:
\begin{equation*}
    g_{m}(uq;p,q) = (1-u)(1-up)\cdots (1-up^m) g_m(u;p,q),
\end{equation*}
and for the summation terms:
\begin{align*}
    & u p^b g_{b}(u;p,q) g_{m-1-b}(up^{b+1}q;p,q) \\
    &= u p^b \frac{1}{(u;q)_\infty (up;q)_\infty \cdots (up^b;q)_\infty} \\
    & \quad \times \frac{(1-up^{b+1})(1-up^{b+2}) \cdots (1-up^{m})}{(up^{b+1};q)_\infty (up^{b+1}p;q)_\infty \cdots (up^{b+1}p^{m-1-b};q)_\infty} \\
    &= u p^b (1-up^{b+1})(1-up^{b+2}) \cdots (1-up^{m}) g_m(u;p,q).
\end{align*}
Thus, dividing \eqref{recurlab} by $g_m(u;p,q)$, the identity to be proven reduces to:
\begin{align*}
    1 &= (1-u)(1-up)\cdots (1-up^m) + u p^m \\
    & \quad + \sum_{b=0}^{m-1} u p^b (1-up^{b+1})(1-up^{b+2}) \cdots (1-up^{m}),
\end{align*}
which is a known algebraic identity. This confirms that $g_m$ satisfies the recurrence, establishing the lemma.
\end{proof}

Finally, $f_0 = g_0 = 1/(u;q)_\infty$, and by Lemmas~\ref{lem:fm} and~\ref{lem:gm} the functions $f_m$ and $g_m$ satisfy the same recurrence \eqref{recurra} for $m \ge 1$. Writing $f_m(u;p,q) = \sum_{n\ge 0} c_n u^n$, this recurrence reads $c_n(1-q^n) = p^m c_{n-1} + \ell_n$ for $n \ge 1$, where $\ell_n$ depends only on $f_0,\dots,f_{m-1}$; together with the value $f_m(0;p,q)=g_m(0;p,q)=1$ it determines $f_m$ uniquely from the lower-index functions, and likewise for $g_m$. By induction on $m$ we conclude $f_m = g_m$ for all $m$, which is \eqref{eqn:qpPoly:Snre4} and hence proves Theorem~\ref{thm:qpPoly:Sn2}.

\medskip

\medskip

\noindent{\bf Acknowledgment.} This work was supported by the National Natural Science Foundation of China.

\bibliographystyle{plain}

\end{document}